\documentclass[11pt]{article}
\usepackage[dvips]{graphicx}
\usepackage{latexsym}
\usepackage{amsmath,amssymb,amsfonts,amsthm}
\usepackage[utf8]{inputenc}
\usepackage{xcolor}
\usepackage{enumerate} 

\newcommand{\ind}{1\!\!1}
\newcommand{\card}{\mbox{Card}}
\newtheorem{theorem}{ \textnormal{\bf{T\scriptsize{HEOREM}}}}
\newtheorem{proposition}{\textnormal{\bf{P\scriptsize{ROPOSITION}}}}

\newtheorem{lemma}{\textnormal{\bf{L\scriptsize{EMMA}}}}

\theoremstyle{definition}
\newtheorem{definition}{\textnormal{\bf{D}\scriptsize{EFINITION}}}
\theoremstyle{remark}

\newtheorem{remark}{\textnormal{\bf{R\scriptsize{EMARKS}}}}

\author{%
  Emmanuel Boissard \footnote{ Institut Math\'ematique de Toulouse, Université Paul Sabatier, 118 Route de Narbonne, F 31062 Toulouse Cedex 9,    France. First-Name.Name@math.univ-toulouse.fr}
 \and %
 Serge Cohen \footnotemark[1]
  \and %
 Thibault Espinasse \footnotemark[1] 
\and %
 James Norris \footnote{Statistical Laboratory, Cambridge University, Centre for Mathematical Sciences
Wilberforce Road, Cambridge CB3 0WB England. j.r.norris@statslab.cam.ac.uk}  }

\title{Diffusivity of a random walk on random walks}

\begin{document}
\maketitle

\begin{abstract} 
  We consider  a random walk $\left(Z^{(1)}_n, \cdots, Z^{(K+1)}_n \right) \in \mathbb{Z}^{K+1}$ with 
the constraint that each coordinate of the walk is at distance one from the following one. 
In this paper, we show that this random walk  is slowed down by a variance factor $\sigma_K^2 = \frac{2}{K+2}$ with respect to the case of the classical simple random walk without constraint. 
\end{abstract}


\noindent \textit{AMS classification (2000)}: 05C81, 60F05.

\section{Presentation of the random walk}

Let $\left(Z^{(1)}_n, \cdots, Z^{(K+1)}_n \right) \in \mathbb{Z}^{K+1}$ denote the heights of $K+1$ 
simple random walks on $\mathbb{Z}$, conditioned on satisfying

\begin{equation}\label{e:cond}
\forall n \in \mathbb{N}, \forall i \in \left[ 1, K\right], \left| Z^{(i+1)}_n - Z^{(i)}_n \right| = 1. 
\end{equation}

More precisely, the random walk is a Markov chain on the state space of $K$-step walks in $\mathbb Z$

\begin{equation*}
 \mathcal{S}_K = \{ (z^{(1)}, \ldots, z^{(K+1)}) \in \mathbb{Z}^{K+1}, \quad \forall i \in \left[ 1, K\right], |z^{(i+1)} - z^{(i)}| = 1 \}
\end{equation*}

\noindent where the next step from $z \in \mathcal{S}_K$ is selected uniformly among the neighbours of $z$
in the usual lattice $\mathbb{Z}^{K+1}$ that belong to $\mathcal{S}_K$. In other words, we consider $K+1$ simple
random walks on the lattice $\mathbb{Z}$ coupled under a shape condition.

As in the case of a simple random walk, the rescaled trajectory of a walker, say $Z^{(1)}$, will converge in law to a Brownian motion. However, it is interesting to note that the constraint between each coordinate only slow down the walk by decreasing its variance.

  Since it is classical to illustrate for our students the simple random walk as the motion of a drunk man,  we can illustrate the previous  mathematical fact by considering the random walk as the motion of a chain of prisoners. It should convince even non mathematicians that the motion of the walk is slowed by the constraint. However it seems very hard to guess the variance from this comparison !

More precisely, denote

\begin{equation*}
\forall t \geq 0, \forall n \in \mathbb{N}, \xi_t^{(n)} = \frac{Z^{(1)}_{\lfloor nt\rfloor}}{\sqrt{n}},
\end{equation*}

where $\lfloor x \rfloor $ is the integer part of the real number $x.$

\begin{theorem} \label{t:var}
 The rescaled random walk $\left(\xi_t^{(n)}\right)_{t \geq 0} $ converges in law, as
 $n$ goes to infinity, to a Brownian motion with variance
\begin{equation*}
\sigma_K^2 = \frac{2}{K+2}. 
\end{equation*}
\end{theorem}

Convergence to the Brownian motion is the  usual invariance principle~:~the noteworthy statement here is that it is possible 
to give an explicit expression for the limit diffusivity of the process, and that its expression is particularly simple.

Our object of interest, the motion of $ Z^{(1)},$ is a non-Markovian process that falls into the
class of random walks with internal structure. Related questions of limit diffusivity for random walks conditioned to respect some geometric shape have been studied in
the literature, under the name of ``spider random walks'' or ``molecular spiders'', see \cite{Antal07}. The computation of the limit
diffusion coefficient is also a central aim there, although the model and methods are different.

Our initial motivation was however more remote. Actually 
we first addressed this question starting from combinatorial 
problems related to $6$-Vertex model in relation with the Razumov 
Stroganov conjecture (See~\cite{Cantini11} for instance). 
The problem can also be related to random graph-homomorphisms
(See~\cite{Yadin07}) or the Square Ice Model (See~\cite {Lieb67}
where $Z^{(1)}$ evolves on a torus.) 

Roughly speaking we can say that, in the literature we read, the 
authors consider questions related to the uniform distribution 
on a sequence of finite graphs $G_K $ and wonder about various asymptotics 
when $ K \to + \infty.$ Later in the article the evolution 
of $Z_n$ will be described as the simple random walk on a graph $G_K .$ Hence on the one hand our problem 
is a very simplified version of the problems stated above,  On the other hand
we were surprised to have  such a simple formula for $\sigma_K^2 = \frac{2}{K+2},$ 
which is true for all $ K $ and not only for $ K \to + \infty.$ We thought in the beginning that
the proof of this fact should be simple but it turns out that, 
although elementary, the tools used to obtain the result are 
more sophisticated than  expected. It is the aim of this note 
to show these tools.

To prove the theorem, we will look for a decomposition
\begin{equation*}
Z^{(1)}_n = M_n + f(Z_{n-1}, Z_n), 
\end{equation*}

where  $(M_n)_{ n \in \mathbb N}$  is a martingale, and where $f : \mathbb{Z}^{K+1} \times \mathbb{Z}^{K + 1} \rightarrow \mathbb R $ is a bounded function. 
We will then show that the following limit exists :

\begin{equation*}
\sigma_K^2 = \lim_{n \rightarrow \infty } \frac1{n} \mathbb{E}\left[M_n^2 \right]. 
\end{equation*}

and that it is indeed the desired diffusivity. The path to this conclusion is akin to classical results for Central Limit Theorems for Markov chains (E.g. \cite{Meyn09}).

We will use another equivalent, albeit more geometric, point of view on this decomposition.
We split the chain in two parts : on the one hand, the motion of one of the walkers, and on the other hand, the relative positions of the walkers (which we call the ``shape''
of the chain at a given time). The latter part is a Markov chain over the state space $\{-1, 1 \}^{K}$ and our quantity of interest
is (almost) an observable of this chain. Computing the martingale decomposition that we wish for amounts to
decomposing a discrete vector field over this new state space into a 
 divergence-free part (corresponding to the martingale part) and a
gradient part (corresponding to the function $f$). Owing to a particular geometric property of this vector field,
for which we coin the term ``stationarity'',
it is indeed possible to perform this calculation explicitly.

\section{The random walk with constraint}
\label{s:graph}
Let us denote  $$\forall n \in \mathbb{N},\forall i \in \left[ 1, K\right], Y_n^{(i)} = Z_n^{(i+1)}- Z_n^{(i)}, $$
and $$Y_n = \left(Y_n^{(1)}, \cdots, Y_n^{(K)} \right) \in \left\{ -1,1\right\}^K. $$ Here $Y_n$ describes the shape  of  $\left(Z^{(1)}_n, \cdots, Z^{(K+1)}_n \right) \in \mathbb{Z}^{K+1}$, i.e. the
position of each  $ Z_n^{(i)} $ relatively to the previous one, and 
belongs to  $V_K = \left\{-1,1\right\}^K$,  whereas $Z^{(1)}_n$  can be seen as  the height of the first walker.
Obviously, the evolution of the chain of walkers may be described by the variables $\left(Z^{(1)}_n,Y_n\right)_{ n \in \mathbb N}$.


For a convenient analysis, we will represent our process as the simple random walk on a (multi)-graph $G_K$, which we define below.

Set $V_K = \{-1, 1 \}^K$ : 
the multi-graph $G_K$ is given as a triplet $(V_K, E_K^+ ,E_K^-)$ where $E_K^+, E_K^- \subset  V_K \times V_K  $ 
are two edge sets, called respectively the set of ``positive`` and ''negative`` edges. A couple $(a, b) \in V_K \times V_K$,
$a \neq b$, belongs to $E_K^{+}$ if the vector $b - a \in \{-2, 0, 2 \}^K$ has nonzero entries of alternating signs, with the first one
negative.
Moreover, $E_K^+$ also contains a loop from each $a \in V_K$ to itself, noted $(a, a)^+$.

Likewise, $E_K^-$ contains those couples $(a, b) \in V_K \times V_K$,
$a \neq b$, such that $b - a \in \{-2, 0, 2 \}^K$ has nonzero entries of alternating sign, with the first one
positive, and self-loops noted $(a, a)^-$ for each $a \in V_K$.

Set $E_K = E_K^+ \cup E_K^-$.
Finally, we consider the following function on $E_K.$
\begin{definition}
  \label{def:A}
Let $A : E_K \rightarrow \{ 1, -1\}$ be the function that takes the value $1$ on $E_K^+$ and $-1$ on $E_K^-$.
Note that we have, for any $a \in V_K$, $A((a,a)^\pm )= \pm 1$.
\end{definition}
\begin{proposition}
 Let $(W^K_n)_{n \geq 0}$ be the simple random walk on $G_K$. 
 The processes $\left(Y_n, Z^{(1)}_n \right)_{ n \in \mathbb N}$ and $\Big(W^K_n, \sum_{j = 1}^n A(W^K_{j-1}, W^K_{j} ) \Big)_{n \in \mathbb{N}}$ have the same distribution.
\end{proposition}

\begin{proof}

It is sufficient to prove that the two Markov chains (taking values in $\left\{-1,1\right\}^K \times \mathbb{Z}$) 
$\left(Y_n, Z^{(1)}_n \right)_{ n \in \mathbb N}$  and 
$\Big(W^K_n, \sum_{j = 1}^n A(W^K_{j-1}, W^K_{j} ) \Big)_{n \in \mathbb{N}}$
have the same transition matrix.

We claim that  
\begin{multline*}
\mathcal{L}\left(Y_{n+1} - Y_n, Z^{(1)}_{n+1}-Z^{(1)}_n| Y_n,Z_n^{(1)}\right) =\\ \mathcal{L} \left(W_{n+1} - W_n, A(W^K_{n}, W^K_{n+1})| W_n, \sum_{j = 1}^n A(W^K_{j-1}, W^K_{j} ) \right).
\end{multline*}

Recall that $ \mathcal{S}_K = \left\{ z \in \mathbb{Z}^{K+1}, \forall i \in [1,K], \left| z^{(i + 1)} - z^{(i)} \right| = 1 \right\}$.
Denote 

\begin{equation*}
\mathcal{A} =  \left\{ (z_1,z_2) \in \mathcal{S}_K^2, \forall i \in [1,K+1], \left| z_1^{(i)} - z_2^{(i)} \right| = 1 \right\}.
\end{equation*}

Since $(Z_n)_{n \in \mathbb{N}}$ and $(W_n)_{n \in \mathbb{N}}$ are both simple random walks, the corresponding transition matrices are given respectively by
\begin{align*}
P^{Z}_{z_1,z_2} &= \frac{1}{\card \left\{ z \in \mathcal{S}, (z_1,z) \in \mathcal{A} \right\}} \ind_{(z_1,z_2) \in \mathcal{A}} 
\\ P^{W}_{w_1,w_2} & = \frac{1}{\card \left\{ w \in \left\{-1,1\right\}^K, (w_1,w) \in E_K \right\}} \ind_{(w_1,w_2) \in E_K}.
\end{align*}

Denote now
 \begin{align*}
  \delta : \; & \mathcal{S} \rightarrow \{-1, 1 \}^K \\
  {} & (z^{(1)}, \ldots, z^{(K + 1)}) \mapsto (z^{(2)} - z^{(1)}, \ldots, z^{(K + 1)} - z^{(K)}),
 \end{align*}

so that $Y_n= \delta(Z_n), n \in \mathbb{N}$.

To prove the proposition, it is sufficient to prove that

\begin{equation*}
(z_1,z_2) \in \mathcal{A} \Leftrightarrow \left| z^{(1)}_{2}-z^{(1)}_1 \right| = 1 \text{ and }  \delta(z_2) - \delta(z_1) \in E_K^{\epsilon},
\end{equation*}

where $\epsilon = - \operatorname{sgn}(z^{(1)}_{2}-z^{(1)}_1)$ (and $\operatorname{sgn}$ is the sign function).

\bigskip

First note that, if $\epsilon = \pm 1$, and if $\forall i \in [1,K+1], z_1^{(i)} =z_2^{(i)} - \epsilon$, then $(z_1,z_2) \in \mathcal{A}$.
Moreover, we have in this case $\delta(z_1) = \delta(z_2)$, and $(\delta(z_1),\delta(z_2)) \in E_K^\epsilon$ as claimed.
On the other hand, if $\delta(z_1) = \delta(z_2)$ and $z_2^{(1)}-z^{(1)}_1 = \epsilon \in \{-1,1\}$, then $\forall i \in [1,K+1], z_1^{(i)} =z_2^{(i)}  + \epsilon$. 

For the chain, it means that,  if $Y_{n+1} = Y_n$ then  $Z^{(1)}_{n+1} - Z^{(1)}_n$ is $+1,-1$ with equal probability $\frac{1}{2}$ independently
of $Y_n,Y_{n+1},Z^{(1)}_n$.

Assume now that  $ z_{1} \neq z_2 $. We will prove that  $(\delta(z_1), \delta(z_1)) \in E_K^\epsilon$, where $\epsilon = z^{(1)}_1-z^{(1)}_2$, i.e. $\delta(z_2)-\delta(z_1) \in \{-2,0,2\}^K$, with nonzero entries of alternating signs, and a first one of the sign of $\epsilon$.

Indeed,  there is a first index $$ \tau_1 = \inf \{ i \in \left[ 1, K\right], \; 
\mbox{such that} \quad  \delta(z_{1})^{(i)} \neq  \delta(z_2)^{(i)} \}. $$

There are two possible cases
$$ 
\delta(z_1)^{(\tau_1)} = +1, \; \delta(z_2)^{(\tau_1)} = -1, \;  z^{(\tau_1)}_{2} = z^{(\tau_1)}_1+1, \; 
z^{(\tau_1+1)}_{2} = z_1^{(\tau_1+1)}-1,$$
$$ 
\delta(z_1)^{(\tau_1)} = -1, \; \delta(z_2)^{(\tau_1)} = +1, \;  z^{(\tau_1)}_{2} = z^{(\tau_1)}_1-1, \; 
z^{(\tau_1+1)}_{2} = z^{(\tau_1+1)}_1+1.$$

In the two cases we have  $ A(\delta(z_1), \delta(z_2)) =   z_2^{(\tau_1)}-z_1^{(\tau_1)} = z_2^{(1)}-z_1^{(1)}.$

Furthermore, we have 
$$\delta(z_2)^{(\tau_1)} - \delta(z_1)^{(\tau_1)} = - 2 A(\delta(z_1),\delta(z_2)).$$

Then if $ \tau_1< K $ let us define
$$ \tau_2 = \inf \{ i \in \left[\tau_1 + 1 , K\right], \; 
\mbox{such that} \quad     \delta(z_2)^{(i)} \neq \delta(z_1)^{(i)}  \},$$
where we set by convention $ \tau_2= K + 1,$ if the condition 
defining the infimum is never satisfied. 

Using the same arguments, we get 
$$\delta(z_2)^{(\tau_2)}-\delta(z_1)^{(\tau_2)} = 2A(\delta(z_1),\delta(z_2)). $$

By induction one can define 
$$ \tau_j = \inf \{ i \in \left[\tau_{j-1} + 1 , K\right], \; 
\mbox{such that} \quad     \delta(z_1)^{(i)}\neq \delta(z_2)^{(i)}   \},$$
until $ \tau_{j-1} \ge K.$ 

We have 
$$\delta(z_2)^{(\tau_j)}-\delta(z_1)^{(\tau_j)} = 2(-1)^{j} A(\delta(z_1),\delta(z_2)).$$

By definition, we get
$(\delta(z_1), \delta(z_2)) \in E_K^\epsilon, $
where $\epsilon = z_1^{(1)}-z_2^{(1)}$.

On the other hand, if $( \delta(z_1), \delta(z_2) ) \in E_K^\epsilon$, where $\epsilon = z^{(1)}_1-z^{(1)}_2 \in \{-1,1\}$, then one can recover explicitly $(z_1,z_2)$ from the definition of $\delta$. Moreover, the condition $ \forall i \in [1,K+1], \left| z_2^{(i)} - z_1^{(i)} \right|=1 $ is implied by the previous arguments (following the definitions of the $\tau_j$).

\end{proof}
%








%






%









%

We may denote, $Y_n \overset{e} \rightarrow Y_{n+1}$ when $e \in E^+$ ($Z^{(1)}_{n+1} = Z^{(1)}_n+1$) and $Y_n \overset{e} \leftarrow Y_{n+1}$ when $e \in E^-$ ($Z^{(1)}_{n+1} = Z^{(1)}_n-1$).

For a general $ a \in V_K $ the previous enumeration of its neighbors is surprisingly 
complicated but we can provide some simple examples.

For instance if  $a  = (1, \ldots, 1)\in V_K,$   has only $K+2$ neighbors:
\begin{align}
\label{neigh-1}
 a &\leftarrow a \nonumber
\\ a & \rightarrow a \nonumber
\\ \forall i \in [1,K], a & \rightarrow (1,\cdots, -1, \cdots,1), 
\end{align}

where the $-1$ is in the $i$-th position.

\bigskip

Note now that the graph $G_K$ can also be described inductively : there are only six following possibilities for $(Z^{(1)}_n,Z^{(2)}_n, Z^{(1)}_{n+1},Z^{(2)}_{n+1})$, described below:

\begin{minipage}[b]{0.45\linewidth}
 \includegraphics[scale=1,keepaspectratio=true]{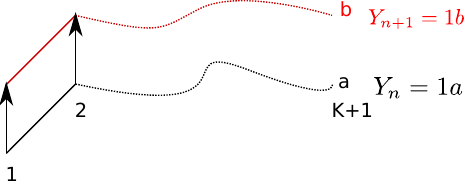}
 \includegraphics[scale=1,keepaspectratio=true]{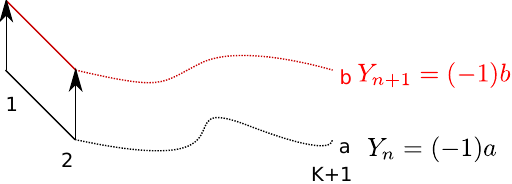}
 \includegraphics[scale=1,keepaspectratio=true]{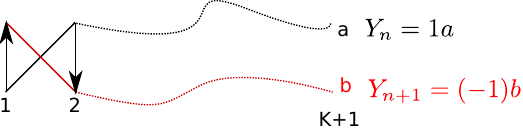}
\end{minipage}
\begin{minipage}[b]{0.45\linewidth}
 \includegraphics[scale=1,keepaspectratio=true]{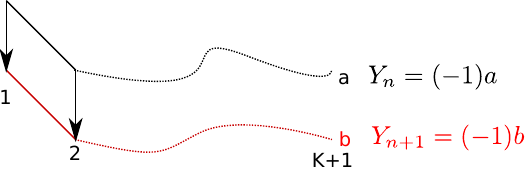}
 \includegraphics[scale=1,keepaspectratio=true]{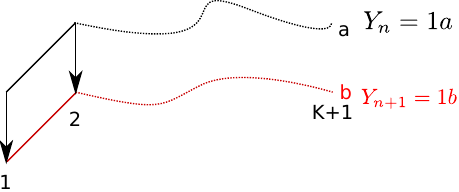}
 \includegraphics[scale=1,keepaspectratio=true]{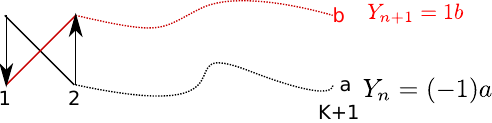}
\end{minipage}

In the figure, we have used the concatenation notation : given a string $a$, the string $1a$, resp. $(-1)a$, is obtained by adding a $1$, resp. $-1$ in front of $a$.
Looking only at the $3$ cases such that $Z^{(1)}_{n+1} - Z^{(1)}_n = 1$, we can deduce  the construction of $G_{K+1}^+$ from $G_K^+$:


 \begin{align*}
V_{K+1} &= \left\{-1,1\right\}^{K+1}\\
E_{K+1}^+ & = \left\{(1a,1b), (a,b) \in E_{K}^+ \right\}  \cup \left\{  ( (-1)a,(-1)b) ,(a,b) \in E_{K}^+ \right\} \cup \left\{ (1a,(-1)b), (a,b) \in E_{K}^- \right\}
 \end{align*}

Figure 1 
shows the first two graphs $G_1^+$ and $G_2^+$. Note that each edge of $G_1$ gives $3$ edges for $G_2$, one on each facet $\left\{ 1a, a \in V_K \right\}$, and $\left\{(-1)a, a \in V_K \right\}$ and one crossing from the facet $\left\{1a, a \in V_K \right\}$ to the facet $\left\{(-1)a, a \in V_K \right\}$.

 \begin{figure}[htbp]
\label{f:constr}
\caption{Construction of $G_1^+$ and $G_2^+$.}
\centering
\includegraphics[scale=0.5,keepaspectratio=true]{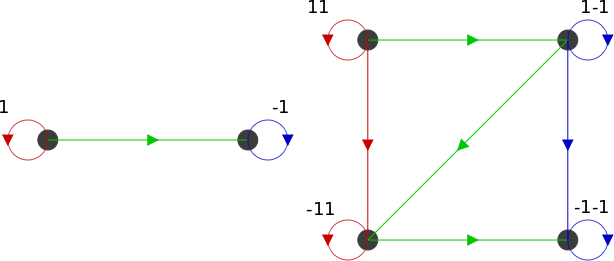}
\end{figure}

We obtain the cardinality $D_K$ of $E_K$ (as a multigraph) by induction:
$$D_K = 2.3^K .$$
We will also make use of the number $\delta_{K}$ of edges of the form $(1a,(-1)b) \in E_K^+$ which can be computed by induction:
$$\delta_K = 3^{K-1}. $$

















Let us now describe vector fields on this graph.

\section{Vector fields on graphs}


In the previous section a function $ A $ has been defined on edges of $ G_K.$ 
We will consider here $ A $ as a vector field on $G_K.$

\subsection{Definitions}
Let us first recall some classical definitions.

\begin{definition}[Vector fields]
 
A vector field on  $G_K = \left(V_K, E_K \right)$ is a function $S: E_K \rightarrow \mathbb{R}$ such that
\begin{equation}
  \label{eq:vect-fi}
  \forall (a,b) \in E_K, a \neq b, S(a,b) = -S(b,a).
\end{equation}
and such that, for any $a \in V_K$, $S((a,a)^+) = -S((a,a)^-)$.
\end{definition}

\begin{definition}[Gradient vector fields]
 
We say that the vector field $S$ on $G_K$ is a gradient vector field if there exists a function $f$
on the vertices of $G_K$ such that for each edge $(a, b) \in E_K$, $S(a, b) = f(b) - f(a)$.
The gradient vector field associated with $f$ is denoted by $\nabla f$.

\end{definition}

\begin{definition}[Divergence and divergence-free vector fields]
 
The divergence of a vector field $S$ at point $x \in G_K$ is defined by

\begin{equation*}
( \nabla \cdot S)(x) = \sum_{ (x,y) \in E_K} S(x, y).
\end{equation*}

We say that a vector field $S$ is divergence-free if its divergence vanishes at all points.

\end{definition}

We can endow the set of vector fields with a scalar product
$$\langle S, S' \rangle_{E_K} = \frac{1}{D_K}\sum_{e \in E_K} S(e)S'(e). $$

Please note that the sum runs over all edges, including loops $(a,a)^\pm$.

Denote also, for any subsets $\phi, \psi$ of $V_K$,
$$J_{\phi,\psi}(S) = \sum_{a \in \phi, b\in \psi, (a,b)\in E_K} S(a,b), $$
the flux  of $S$ going from $\phi$ to $\psi$. Note that a divergence free field $B$ verifies
$$\forall \phi \subset V_K, J_{\phi, V_K \setminus \phi}(B) = 0. $$

\subsection{Hodge decomposition of vector fields on graphs}


In  analogy with the case of vector fields in Euclidean spaces, we can decompose any vector field on 
$G_K$ into the sum of a gradient vector field and a divergence-free field. The following proposition is well-known.

\begin{proposition} \label{prop_dec_usual}
 Let $S$ be a vector field on $G_K$. There exist a unique gradient vector field $\nabla f$ and a unique divergence-free
field $B$ such that 

\begin{equation*}
 S = \nabla f + B.
\end{equation*}

Moreover $\langle \nabla f, B \rangle_{E_K}=0  $ 
\end{proposition}

The last identity simply means that gradient fields and divergence-free fields are orthogonal complements of each other in the
vector space of vector fields over $G_K$.

In our case we are interested in stationary vector fields. 

\begin{definition}[Stationary vector field]
\label{def:stat}
A subgraph $G$ of the complete graph on  $\{-1, 1 \}^K$ is 
\emph{stationary}, if the following holds. For  $u, u' \in  \{-1, 1 \}^K$ and  $v \in \mathbb R^k$ such that $u + v, u' + v \in \{-1, 1 \}^K,$ 
if $(u, u + v)$ is an edge of $G$, then $(u', u' + v)$ is an edge of $G$.

 A vector field $S$ defined on a subgraph of
the complete graph on $\{-1, 1 \}^K$ with a stationary domain   is \emph{stationary}
if for all $(u, v)$ edges of $G$, $S(u, v)$ only depends on $u - v$
(where $\{-1, 1 \}^K$ is embedded in $\mathbb{R}^K$ in an
obvious way).
\end{definition}
\begin{remark}
Note that, thanks to the construction of $E_K^+$, it is stationary. 
Moreover if $(u,u+v) \in E_K^+ $ and $u',u'+v \in V_K$,
then $(u',u'+v) \in E_K^+$ and thus the vector field $A$ taking values
$1$, resp. $-1$, on $ E_K^+$, resp. $E_K^-$, is stationary.
So we may  expect the gradient vector field $\nabla f $ in the Hodge 
decomposition of $ A $ to be stationary. Unfortunately if $S$ is 
 a stationary vector field on $G_K$ and its decomposition is 
$$
 S = \nabla f + B
$$
as per Proposition \ref{prop_dec_usual}, then $\nabla f$  is not always stationary.


Nevertheless it turns out that the gradient vector field $\nabla f $ in the Hodge 
decomposition of $ A $ is actually stationary as it will be shown 
in the next section.  
\end{remark}

\subsection{Hodge decomposition of $A$}
\label{sec:hodge-A}
Let us recall the Definition~\ref{def:A}   the vector field $A$ on $G_K$  is such that 
\begin{align*}
 \forall e \in E_K^+, A(e) = 1
\\ \forall e \in E_K^-, A(e) = -1.
\end{align*}
In this section our aim is to compute a function $ f $ such that 
\begin{equation}
  \label{eq:def-gradA}
  ( \nabla \cdot A) (a) = ( \nabla \cdot (\nabla f)) (a) \quad \forall a  \in V_K.
\end{equation}
One can first remark that $ \forall a  \in V_K$
\begin{equation}
  \label{eq:nablaA}
  ( \nabla \cdot A) (a) = \quad \card\{b \in V_K,\; (a,b) \in E^+_K\} - \card\{b \in V_K,\; (a,b) \in E^-_K\}.
\end{equation}
In the previous equation we used the notation $\card$ for cardinality of sets. 
We will introduce various notations related to other cardinalities
\begin{eqnarray}
  \alpha(a) &=&  \card\{b \in V_K,\; (a,b) \in E^+_K\} \\
  \bar{\alpha}(a) &=&  \card\{b \in V_K,\; (a,b) \in E^-_K\}.
\end{eqnarray}
Then  we need also to define for $ 0 \le k \le K,$ 
$ \alpha_k(a) $ the number of the vertices $ b $ such that 
$ (a,b) \in E^+_K$ with $ k $  digits $ i \in \{1,\ldots,K\} $ 
such that $ a_i \neq b_i .$ Similarly 
$\bar{\alpha}_k(a) $ is the number of the vertices $ b $ such that 
$ (a,b) \in E^-_K$ and  $ k $  digits $ i \in \{1,\ldots,K\} $ 
such that $ a_i \neq b_i.$ 
Then we consider $\alpha_{ev}(a) = \sum_{k \; \; even}\alpha_k(a) $
and $\alpha_{od}(a) = \sum_{k \; \; odd} \alpha_k(a).$
Similarly $\bar{\alpha}_{ev}(a) = \sum_{k \; \; even}\bar{\alpha}_k(a) $
and $\bar{\alpha}_{od}(a) = \sum_{k \; \; odd} \bar{\alpha}_k(a).$

 Let us consider the function $ f_1 $ on $ V_K$ 
such that 
\begin{equation}
  \label{eq:f1}
f_1(a)= \alpha_1(a) - \bar{\alpha}_1(a) = \sum_{i=1}^K a_i.
\end{equation}
The last equation is trivial since $ \alpha_1(a)$ is the number
of the $ a_i $s equal to $ 1 $ and $ \bar{\alpha}_1(a) $
the number of $ a_i $s equal to $ -1.$ Obviously we also get
\begin{equation}
  \label{eq:sum-alpha1}
  \alpha_1(a) + \bar{\alpha}_1(a)= K 
\end{equation}
for any vertex $a$ in $V_K.$
Let us remark that for any function $ f $ on   $V_K$
$$ ( \nabla \cdot (\nabla f)) (a) = \sum_{(a,b) \in E_K} f(b) - f(a).$$
Since $ f_1$ yields the sum of the digits of any vertex, we first observe that 
if $ (a,b) \in E_K $ and the number of digits $ i \in \{1,\ldots,K\} $ 
such that $ a_i \neq b_i $ is even then $ f_1(b)-f_1(a) = 0.$
If the number of digits $ i \in \{1,\ldots,K\} $ 
such that $ a_i \neq b_i $ is odd and  $ (a,b) \in E_K^+$ then 
 $ f_1(b)-f_1(a) = -2.$ One can then deduce that 
\begin{equation}
  \label{eq:nabla-grad-f1}
  ( \nabla \cdot (\nabla f_1)) (a)= - 2 (\alpha_{od}(a) - \bar{\alpha}_{od}(a)).
\end{equation}

 It turns out that if we consider the function $ f_2$ on $ V_K$ 
such that 
\begin{equation}
  \label{eq:f2}
f_2(a)= \alpha_2(a) - \bar{\alpha}_2(a),
\end{equation}
then 
 \begin{equation}
  \label{eq:nabla-grad-f2}
  ( \nabla \cdot (\nabla f_2)) (a)= - (K+2) (\alpha_{ev}(a) - \bar{\alpha}_{ev}(a)).
\end{equation}
We will prove~\eqref{eq:nabla-grad-f2} by induction on $ K.$
To do that we split $ ( \nabla \cdot (\nabla f_2)) $ into the sum of two 
functions
\begin{eqnarray}
  \label{eq:phi} \phi(a) &=& \sum_{(a,b) \in E^+_K} f_2(b) - f_2(a)\\
  \label{eq:barphi} \bar{\phi}(a) &=& \sum_{(a,b) \in E^-_K} f_2(b) - f_2(a). 
\end{eqnarray}
To proceed the induction argument we remark that for any vertex $a$ in $V_K$
\begin{eqnarray}
  \label{eq:f1+} f_1(1a) &=& f_1(a) + 1\\
  \label{eq:f1-} f_1(-1a) &=& f_1(a) - 1 \\ 
  \label{eq:f2+} f_2(1a) &=& f_2(a) +  \bar{\alpha}_1(a) \\
  \label{eq:f2-} f_2(-1a) &=& f_2(a) -  \alpha_1(a).
\end{eqnarray}
We will then compute $ \phi(1a),\;  \phi(-1a),\; \bar{\phi}(1a), \; \bar{\phi}(-1a).$
The easiest computation is 
\begin{eqnarray*}
  \bar{\phi}(1a) &=&  \sum_{(1a,b) \in E^-_{K+1}} f_2(b) - f_2(1a)\\
&=&\sum_{(a,c) \in E^-_K}f_2(1c) - f_2(1a),
\end{eqnarray*}
since $(1a,-1c) \in E^+_{K+1}.$ Then 
\begin{eqnarray}
\bar{\phi}(1a) &=&\sum_{(a,c) \in E^-_K}f_2(c) - f_2(a) + \sum_{(a,c) \in E^-_K}
\bar{\alpha}_1(c) - \bar{\alpha}_1(a). \nonumber 
\end{eqnarray}
To evaluate $ \sum_{(a,c) \in E^-_K}
\bar{\alpha}_1(c) - \bar{\alpha}_1(a) $ we use that $ \bar{\alpha}_1(a) $ is the number of 
digits equal to $ -1 $ in $ a.$ If $ (a,c) \in E^-_K$ and if the number 
of $ a_i \neq c_i $ is even then $ \bar{\alpha}_1(c)= \bar{\alpha}_1(a).$
If this number is odd then $ \bar{\alpha}_1(c)= \bar{\alpha}_1(a)-1.$
Therefore $ \sum_{(a,c) \in E^-_K} \bar{\alpha}_1(c) - \bar{\alpha}_1(a) = - \bar{\alpha}_{od}(a).$
Hence
\begin{eqnarray}
\bar{\phi}(1a) &=& \bar{\phi}(a)- \bar{\alpha}_{od}(a)  \label{eq:rec-barphi1}.
\end{eqnarray}
One also get in the same way
\begin{equation}
  \label{eq:rec-phi-1}
\phi(-1a)= \phi(a)- \alpha_{od}(a).
\end{equation}
The induction is a bit more involved for 
\begin{eqnarray*}
\phi(1a) &=&  \sum_{(1a,b) \in E^+_{K+1}} f_2(b) - f_2(1a)\\
&=&  \sum_{(1a,1c) \in E^+_{K+1}} f_2(1c) - f_2(1a) + 
\sum_{(1a,-1c) \in E^+_{K+1}} f_2(-1c) - f_2(1a).
\end{eqnarray*}
Because of~\eqref{eq:f2+}
\begin{eqnarray*}
 \sum_{(1a,1c) \in E^+_{K+1}} f_2(1c) - f_2(1a) &=& \sum_{(a,c) \in E^+_K} f_2(c) - f_2(a) + \sum_{(a,c) \in E^+_K}\bar{\alpha}_1(c) -\bar{\alpha}_1(a) \\
&=& \phi(a) + \alpha_{od}(a).
\end{eqnarray*}
Then 
\begin{eqnarray*}
\sum_{(1a,-1c) \in E^+_{K+1}} f_2(-1c) - f_2(1a)&=& \sum_{(a,c) \in E^-_K} f_2(c) - f_2(a) - \sum_{(a,c) \in E^-_K}\alpha_1(c)  + \bar{\alpha}_1(a) \\
&=& \bar{\phi}(a) - \sum_{(a,c) \in E^-_K}\alpha_1(c)  + \bar{\alpha}_1(c)\\
&& \phantom{\bar{\phi}(a)}+ \sum_{(a,c) \in E^+_K} \bar{\alpha}_1(c) - \bar{\alpha}_1(a)\\
&=& \bar{\phi}(a) - K \bar{\alpha}(a) -\bar{\alpha}_{od}(a). 
\end{eqnarray*}
Hence
\begin{equation}
  \label{eq:rec-phi1}
\phi(1a)= \phi(a) + \bar{\phi}(a) + \alpha_{od}(a) - 
\bar{\alpha}_{od}(a) - K \bar{\alpha}(a).
\end{equation}
Similarly we get 
\begin{equation}
  \label{eq:rec-barphi-1}
\bar{\phi}(-1a)= \phi(a) + \bar{\phi}(a) - \alpha_{od}(a)  +
\bar{\alpha}_{od}(a) + K \bar{\alpha}(a).
\end{equation}
We can now evaluate the functions $\phi, \; \bar{\phi}.$
\begin{lemma}
\label{lem:phi_barphi}
$\forall a \in V_K,$ 
\begin{equation}
  \label{eq:phi}
  \phi(a) = \bar{\alpha}_{ev}(a)-(K+1) \alpha_{ev}(a) + \alpha_{od}(a) + \bar{\alpha}_{od}(a)
\end{equation}
\begin{equation}
  \label{eq:barphi}
  \bar{\phi}(a) = -\alpha_{ev}(a)+(K+1) \bar{\alpha}_{ev}(a) - \bar{\alpha}_{od}(a) - \alpha_{od}(a) 
\end{equation}
\end{lemma}
\begin{proof}
We will only sketch the proof performed by an easy induction on $ K $ for $\phi,$
computations are similar for $\bar{\phi}.$ Let us assume that~\eqref{eq:phi}, \eqref{eq:barphi}
hold for $ K,$ we have to compute $ \phi(1a) $ and $ \phi(-1a) $ and check that 
they fulfill ~\eqref{eq:phi}, \eqref{eq:barphi} for $ K+1.$
Because of~\eqref{eq:rec-phi1}
$$ \phi(1a) = \phi(a) + \bar{\phi}(a) + \alpha_{od}(a) - 
\bar{\alpha}_{od}(a) - K \bar{\alpha}_{ev}(a)  - K \bar{\alpha}_{od}(a).$$
One can check that $ \bar{\alpha}_{ev}(1a) = \bar{\alpha}_{ev}(a),\; \alpha_{ev}(1a)= \alpha_{ev}(a) + \bar{\alpha}_{od}(a),\;
 \alpha_{od}(1a)=  \alpha_{od}(a) + \bar{\alpha}_{ev}(a) + 1,\; \bar{\alpha}_{od}(1a)=\bar{\alpha}_{od}(a).$
Using~\eqref{eq:phi} for $ K,$  we get~\eqref{eq:phi} for $ K+1$ and $ \phi(1a).$ The computations 
for  $ \phi(-1a) $ are left to the reader.
\end{proof}

 By summing~\eqref{eq:phi} and~\eqref{eq:barphi} we get~\eqref{eq:nabla-grad-f2}, and 
we deduce that if we take 
\begin{equation}
  \label{eq:gradA}
  f = - \left ( \frac12 f_1 + \frac1{K+2} f_2 \right ) + \frac{K}2,
\end{equation}
$ A - \nabla f $ is divergence free. Please note that obviously the additive constant 
in~\eqref{eq:gradA} is arbitrary but it yields the following convenient 
expression  of $f(a) $ in terms of the digits of $ a \in V_K.$
\begin{lemma}
\label{lem:form-grad A}
\begin{equation}
  \label{eq:gradA-expli}
  f(a) = \sum_{i,\; a_i= -1} F_i^K,
\end{equation}
where 
\begin{equation}
  \label{eqFiK}
 F_i^K= \frac{3 + 2(K-i)}{K+2}.
\end{equation}
Moreover $\nabla f $ is a stationary gradient vector field. 
\end{lemma} 
\begin{proof}
Since we already know that $f_1(a)= \sum_{i=1}^K a_i,$
it is enough to  show by induction on  $ K $ that 
\begin{equation}
  \label{eq:f2-expli}
  f_2(a) = \frac12 \sum_{i=1}^K a_i (1 + K - 2 i).
\end{equation}
Obviously the formula is true for $ K = 1.$
Because of~\eqref{eq:f2+} and~\eqref{eq:f2-}, 
\begin{eqnarray*}
  f_2(1a)+ f_2(-1a) &=& 2 f_2(a) - f_1(a) \\
  f_2(1a)- f_2(-1a) &=& K. 
\end{eqnarray*}
Hence~\eqref{eq:f2-expli} is proved for 
$K+1.$
Owing to $f_1(a)= \sum_{i=1}^K a_i,$ and~\eqref{eq:f2-expli}, 
it is obvious that $\nabla f_1, \; \nabla f_2 $ are stationary and 
consequently $\nabla f $ is a stationary gradient vector field.
\end{proof}

\subsection{Proof of theorem \ref{t:var}}




We denote by $(\nabla f, B)$ the  decomposition of $ A $ as per Proposition \ref{prop_dec_usual}.

Back to the original problem, we recall that
$$\forall n \geq 0,Z^{(1)}_{n+1} - Z^{(1)}_n  = B(Y_{n},Y_{n+1})+ \nabla f(Y_n, Y_{n+1}).$$

Let us denote $M_n = Z^{(1)}_n - f(Y_n)$. Let  $\mathcal{F}_n=\sigma ((Y_k,Z^{(1)}_k), \; k \le n) $, we have

$$\mathbb{E}[M_{n+1} - M_n | \mathcal{F}_n ] = (\nabla \cdot B)(Y_n) = 0, $$
and $ (M_n)_{n \in \mathbb N} $ is a martingale. 

We now sketch out how to apply the Central Limit Theorem for Markov chains. Let $e_n=(Y_n,Y_{n+1}),$
$(e_n),{n \geq 0}$ is a Markov chain on $E_K.$ 
Then our quantity of interest $Z_n^{(1)}$ is an additive
observable of the process $(e_n),{n \geq 0}$, as

\begin{equation*}
Z_{n}^{(1)} - Z_0^{(1)} = \sum_{k = 0}^{n-1} Z_{k + 1}^{(1)} - Z_{k}^{(1)} = \sum_{k = 0}^{n - 1} A(e_k).
\end{equation*}

The Central Limit Theorem for Markov chains (see e.g. \cite{Meyn09}) shows that 
$(\frac{1}{\sqrt{n}} Z^{(1)}_{\lfloor nt \rfloor})_{t \geq 0}$
converges as $n \rightarrow + \infty$ to a Brownian motion, with variance given by 

\begin{align*}
 \sigma_K^2 &= \lim_{n \rightarrow + \infty} \frac{1}{n} \mathbb{E} \left[\left(Z^{(1)}_n\right)^2\right] \\
            &=  \lim_{n \rightarrow + \infty} \frac{1}{n} \mathbb{E} \left[M_n^2\right]+
\lim_{n \rightarrow + \infty} \frac{1}{n} \mathbb{E} \left[ f(Y_n)^2\right] +
\lim_{n \rightarrow + \infty} \frac{2}{n} \mathbb{E} \left[\left(f(Y_n) M_n \right)\right].
\end{align*} 

 Let us first compute $ \lim_{n \rightarrow + \infty} \frac{1}{n} \mathbb{E} \left[M_n^2\right].$
We can remark that $ M_n^2 - \sum_{i = 0}^{n-1} B(Y_i,Y_{i+1})^2 $ 
is a $ \mathcal{F}_n $ martingale, hence 
$  \lim_{n \rightarrow + \infty} \frac{1}{n} \mathbb{E} \left[M_n^2\right]=\lim_{n \rightarrow + \infty} \frac{1}{n} \mathbb{E} \left[ \sum_{i = 0}^{n-1} \Big(B(Y_i,Y_{i+1})   \Big)^2\right].$

 
If $\mu$ denotes the invariant measure for the random walk $(Y_n)_{n \geq 0}$, by ergodicity, we get
$$ \lim_{n \rightarrow + \infty} \frac{1}{n} \mathbb{E} \left[ \sum_{i = 0}^{n-1} \Big(B(Y_i,Y_{i+1})   \Big)^2\right]=\mathbb{E}_\mu \left[\left(B(Y_0, Y_1)\right)^2\right].$$

Under $ \mu $ the distribution of $(Y_0,Y_1)$ is uniform on $ E_K$ because
$ Z_1 $ is  uniformly chosen among all neighbours of $Z_0,$ hence 
\begin{equation}
\label{eq:inc-lim} 
 \lim_{n \rightarrow + \infty}\frac{1}{n} \mathbb{E}(M_n^2) = \| B\|^2,
\end{equation}
Using that  $f$ does not depend on $n,$
we obtain
$$ \lim_{n \rightarrow + \infty} \frac{1}{n} \mathbb{E} \left[ f(Y_n)^2\right] = 0,$$
and by Cauchy Schwarz inequality and~\eqref{eq:inc-lim}
 \begin{equation*}
  \lim_{n \rightarrow + \infty} \frac{2}{n} \mathbb{E} \left[\left(f(Y_n) M_n \right)\right]=0.
\end{equation*} 

So we get
 \begin{equation*}
 \sigma_K^2 = \lim_{n \rightarrow + \infty} \frac{1}{n} \mathbb{E} \left[ \sum_{i = 0}^{n-1} \Big(B(Y_i,Y_{i+1})   \Big)^2\right],
\end{equation*}

and  by ergodicity, 

\begin{equation*}
 \sigma_K^2 = \mathbb{E}_\mu \left[\left(B(Y_0, Y_1)\right)^2\right].
\end{equation*}

Since, under  $ \mu,$ the distribution of $(Y_0,Y_1)$ is uniform on $ E_K$ 
\begin{equation*}
 \sigma_K^2 = \left\| B \right\|^2.
\end{equation*}
 
By orthogonality of $B$ and $\nabla f$,

\begin{equation*}
 \sigma_K^2 = \left\| A \right\|^2 - \langle A, \nabla f \rangle.
\end{equation*}

Thus, it remains  to compute $\langle A, \nabla f \rangle$ (since $\left\|A \right\| = 1$, by definition).

At this point we use the fact  that $\nabla f$ is  a
stationary  field, in the sense of Definition~\ref{def:stat}.
Then if we denote by 
\begin{equation*}
a = (1, \ldots, 1), \quad z_i = (1, \ldots, -1, 1, \ldots, 1)
\end{equation*}
(where $-1$ is in $i$-th position) and, 
for $1 \leq i \leq K$, we have by~\eqref{eq:gradA-expli}
\begin{equation*}
 F_i = \nabla f(a, z_i),
\end{equation*}

Now  we compute $\langle A, \nabla f \rangle$ as a  function of $F_1$, the value $\nabla f$ on the edge $(1,\cdots, 1) \rightarrow (-1,1, \cdots, 1)$.

By~\eqref{eq:vect-fi}
$$  \langle A, \nabla f \rangle   = \frac{2}{D_K}\sum_{(a,b) \in E_K^+} (\nabla f)(a,b). $$
If $\phi = \left\{ 1a', a' \in V_{K-1} \right\}$ and $\psi = \left\{( -1)b', b' \in V_{K-1} \right\}$, then using the relation between $ E_{K-1} $ and $ E_K,$  
$$  \langle A, \nabla f \rangle   =   \frac{2}{D_K}\sum_{(a,b) \in E_K^+\cap \phi } (\nabla f)(a,b) + \frac{2}{D_K}\sum_{(a,b) \in E_K^+\cap \psi } (\nabla f)(a,b) + \frac{2}{D_K}J_{\phi,\psi}(\nabla f).$$
By stationarity of $  \nabla f, $  
\begin{align*}
\langle A, \nabla f \rangle  &= \frac{4}{D_K}\sum_{(a,b) \in E_K^+\cap \phi } (\nabla f)(a,b) + \frac{2}{D_K}J_{\phi,\psi}(\nabla f) 
\\
& = \frac{4}{D_K}\sum_{(a,b) \in E_K^+\cap \phi } (\nabla f)(a,b) + \frac{4}{D_K}J_{\phi,\psi}(\nabla f) - \frac{2}{D_K} J_{\phi,\psi}(\nabla f).
\end{align*}


Then, because of the definition of $\phi $ and $ \psi,$  
\begin{align*}
\langle A, \nabla f \rangle  & = \frac{4}{D_K}\sum_{(a', b') \in E_{K-1}^+} (\nabla f)(1a',1b') + (\nabla f)(1b',(-1)a') - \frac{2}{D_K}J_{\phi,V_K \setminus \phi}(\nabla f),
\\ &= \frac{4}{D_K}\sum_{(a',\,b') \in E_{K-1}^+} f(1b')- f(1a')+f((-1)a') -f(1b') - \frac{2}{D_K}J_{\phi,V_K \setminus \phi}(A).
\end{align*}
Then by stationarity of  $\nabla f$
\begin{align*}
\langle A, \nabla f \rangle   &=4 \frac{\card E_{K-1}^+}{D_K}F_1 - 2\frac{\delta_K}{D_K}
\\ &= \frac2{K+2}. 
\end{align*}






\begin{remark}
In the proof, the way we guessed~\eqref{eqFiK} is a bit mysterious. 
Assuming that $\nabla f $ is a stationary gradient vector field, the family $(F_i)_{i \in [1,K]}$ can be computed considering the system of equations given by
$$\forall i \in [1,K], J_{M_i,N_i } = 0, $$
where $M_i = \left\{ (a_j)_{j \in [1,K]} \in V_K, a_i =1\right\}$ and $N_i = \left\{ (a_j)_{j \in [1,K]} \in V_K, a_i =-1\right\}$.

This leads to the following system:
 \begin{equation*}
 \left[ \begin{matrix}3^{K-1} & -3^{K-2} & \ddots & -3& -1
 \\                   -3^{K-2} & 3^{K-1}  & \ddots & -9 & -3
 \\                  \ddots & \ddots & \ddots & \ddots & \ddots 
 \\              -3 &           -9 & \ddots & 3^{K-1} & -3^{K-2}
 \\                -1 & -3 & \ddots & -3^{K-2} & 3^{K-1}\end{matrix} \right]   F = \left[\begin{matrix}3^{K-1}
 \\                   3^{K-2}
 \\                  \vdots 
 \\ 3
 \\                1\end{matrix} \right].
 \end{equation*}
 
 The unique solution is given by :
 
 \begin{equation*}
 \forall j \leq K, F_j = \frac{3+2 (K-j)}{K+2}.
 \end{equation*}
Even if the guess is correct, we did not find another way as the techniques used in the Section~\ref{sec:hodge-A}
to show that $\nabla f $ is a stationary gradient vector field. 
\end{remark}

\section{Some further questions}

In this final section we briefly outline some related problems.

\begin{itemize}
 \item It is possible to make sense of the process when $K$ is infinite. Several questions arise : what happens to the process of one marked walker~? 
Is there a scaling limit under equilibrium for the ``shape''process $(Z^{(k)} - Z^{(1)})_{1 \leq k}$ ?

Another natural step would be to let $K$ grow with $N$ in a suitable way, so as to get a scaling limit for the two-parameter process
$(Z_n^{(k)} - Z_n^{(1)})_{1 \leq k \leq K+1, n \in \mathbb{Z}}$.

\item One may also ask about different quantities, such as the diameter of the set of walkers under the invariant measure for the entire walk.

\item One may also consider random walkers conditioned on satisfying different shape constraints, and on graphs more general than $\mathbb{Z}$.
 As a starting example, what happens if we work on a torus, i.e. if we force also $|Z^{(K+1)}_n - Z^{(1)}_n| = 1$ ? The ''shape'' chain changes in this case and it is no longer irreducible over $\{-1, 1 \}^{K+1}$
(one may check that the number of $-1$ symbols is fixed, and that this enumerates the recurrence classes). It is interesting to point out that
this setup is the one chosen by E. Lieb for the computation of the ''six-vertex constant`` in \cite{Lieb67}.
\end{itemize}

\section{Acknowledgments}

We warmly thank Charles Bordenave for fruitful conversations and useful comments.

\def\cprime{$'$}
\providecommand{\bysame}{\leavevmode\hbox to3em{\hrulefill}\thinspace}
\providecommand{\MR}{\relax\ifhmode\unskip\space\fi MR }
\providecommand{\MRhref}[2]{%
  \href{http://www.ams.org/mathscinet-getitem?mr=#1}{#2}
}
\providecommand{\href}[2]{#2}







\begin{thebibliography}{1}

\bibitem{Yadin07}
Itai Benjamini, Ariel Yadin, and Amir Yehudayoff, \emph{Random
  graph-homomorphisms and logarithmic degree}, Electron. J. Probab. \textbf{12}
  (2007), no. 32, 926--950. \MR{2324796 (2008f:60012)}

\bibitem{Cantini11}
Luigi Cantini and Andrea Sportiello, \emph{Proof of the {R}azumov-{S}troganov
  conjecture}, J. Combin. Theory Ser. A \textbf{118} (2011), no.~5, 1549--1574.
  \MR{2771600}

\bibitem{Lieb67}
E.~H. Lieb, \emph{Residual entropy of square ice}, Physical Review \textbf{162}
  (1967), no.~1, 162--172.

\bibitem{Meyn09}
Sean Meyn and Richard~L. Tweedie, \emph{Markov chains and stochastic
  stability}, second ed., Cambridge University Press, Cambridge, 2009, With a
  prologue by Peter W. Glynn. \MR{2509253 (2010h:60206)}
  
 \bibitem{Antal07}
 Tibor Antal et al, \emph{Molecular Spiders in One Dimension}, J. Stat. Mech. (2007) P08027

\end{thebibliography}
\end{document}